\documentclass[12pt,dvipsnames]{amsart}
\usepackage[colorlinks=true, pdfstartview=FitV, linkcolor=blue, citecolor=blue, urlcolor=blue]{hyperref}
\usepackage{geometry,pb-diagram}
\usepackage{amssymb, amscd, amsmath, amsfonts, amsthm}
\usepackage{stmaryrd}
\usepackage{verbatim}
\usepackage{multicol}
\usepackage{textcomp}

\newcommand{\idiot}[1]{\vspace{5 mm}\par \noindent
\marginpar{\textsc{For longer version}}
\framebox{\begin{minipage}[c]{.99 \textwidth}
#1 \end{minipage}}\vspace{5 mm}\par}

\newcommand{\rant}[1]{\vspace{5 mm}\par \noindent
\marginpar{\textsc{Comments}}
\framebox{\begin{minipage}[c]{.99 \textwidth}
\tt #1 \end{minipage}}\vspace{5 mm}\par}

\renewcommand{\idiot}[1]{}
\renewcommand{\rant}[1]{}

\def\unprotectedboldentry#1{\textcolor{Red}{\textbf{#1}}}
\def\boldentry{\protect\unprotectedboldentry}
\usepackage{tikz}
\usetikzlibrary{calc}
\usepackage{ifthen}
\newcommand{\tikztableau}[2][scale=0.6,every node/.style={font=\small}]{
    \def\newtableau{#2}
    \begin{array}{c}
    \begin{tikzpicture}[#1]
    \coordinate (x) at (-0.5,0.5);
    \coordinate (y) at (-0.5,0.5);
    \foreach \row in \newtableau {
        \coordinate (x) at ($(x)-(0,1)$);
        \coordinate (y) at (x);
        \foreach \entry in \row {
            \ifthenelse{\equal{\entry}{X}}
               {
                \node (y) at ($(y) + (1,0)$) {};
                \fill[color=gray!10] ($(y)-(0.5,0.5)$) rectangle +(1,1);
                \draw[color=gray] ($(y)-(0.5,0.5)$) rectangle +(1,1);
               }
               {
                \ifthenelse{\equal{\entry}{\boldentry X}}
                   {
                    \node (y) at ($(y) + (1,0)$) {};
                    \fill[color=gray] ($(y)-(0.5,0.5)$) rectangle +(1,1);
                    \draw ($(y)-(0.5,0.5)$) rectangle +(1,1);
                   }
                   {
                    \node (y) at ($(y) + (1,0)$) {\entry};
                    \draw ($(y)-(0.5,0.5)$) rectangle +(1,1);
                   }
               }
            }
        }
    \end{tikzpicture}
    \end{array}}
\newcommand{\tikztableausmall}[1]{\tikztableau[scale=0.45,every node/.style={font=\rm\small}]{#1}}

\def\sym{\operatorname{\mathsf{Sym}}}

\def\Qsym{\operatorname{\mathsf{QSym}}}

\def \fS{{\mathfrak S}}
\def \HH{{H}}
\def\Nsym{\operatorname{\mathsf{NSym}}}

\newtheorem{Theorem}{Theorem}[section]

\newtheorem{Corollary}[Theorem]{Corollary}
\newtheorem{Lemma}[Theorem]{Lemma}
\newtheorem{Example}[Theorem]{Example}
\theoremstyle{definition}

\newtheorem{Definition}[Theorem]{Definition}

\begin{document}
\title{Structure Constants for Immaculate Functions}
\author[Shu Xiao Li]{Shu Xiao Li}
\address[Shu Xiao Li]{Department of Mathematics and Statistics\\ York  University\\       To\-ron\-to, Ontario M3J 1P3\\ CANADA} \email{lishu3@yorku.ca}
\date{\today}
 
\begin{abstract}
The immaculate functions, $\fS_{\alpha}$, were introduced as a Schur-like basis for $\Nsym$. We investigate facts about their structure constants. These are analogues of Littlewood-Richardson coefficents. We will give a new proof of the left Pieri rule for
the $\fS_{\alpha}$, a translation invariance property for the structure coefficients of the
$\fS_{\alpha}$, and a counterexample to an $\fS_{\alpha}$-analogue of the saturation conjecture.
\end{abstract}

\maketitle
\setcounter{tocdepth}{3}

\section{Introduction}

\noindent
Symmetric functions, along with their foundamental Schur basis, are one of the most important objects in algebraic combinatorics. They appear in numerous areas throughout mathematics including representation theory \cite{Sagan}, algebraic geometry \cite{Manivel} and much more. There have been intensive studies on symmetric functions and many generalizations have been developed such as non-commutative symmetric functions ($\Nsym$).\\
\\
The notion of $\Nsym$, as a Hopf algebra analogue of the symmetric functions, was first introduced by I.M. Gelfand, D. Krob, et al. in 1995 \cite{GKLLRT}. Then, it showed great importance in algebraic combinatorics, together with its dual, $\Qsym$. Symmetric functions, $\Nsym$ and $\Qsym$ are known to be universal objects in a certain category of Hopf algebra \cite{ABS}. In \cite{BBSSZ2}, the authors constructed a Schur-like basis for $\Nsym$, called immaculate functions, $\{\fS_{\alpha}\}$ indexed by compositions. They are defined in a similar way to Schur functions, they have a lot of properties analogous to the Schur basis and their images under the forgetful projection are Schur functions.\\
\\
In this paper, we will focus on the immaculate version of Littlewood-Richardson coefficients, that is the multiplicative structure constants of immaculate functions. In $\sym$, these coefficients are obtained from the expansion of the product of two Schur functions in Schur basis. They are also linked with other branches of mathematics such as the decomposition of tensor product of Schur modules and the saturation conjecture (\cite{KT}). They are all positive and LR rule (Theorem \ref{thm:2.1}) gives them a combinatorial interpretation.\\
\\
For immaculate functions, \cite{BBSSZ2} gave a right Pieri rule (\ref{thm:2.3}) that is positive and multiplicity free. However, the structure constants could be negative in general. For instance, in \cite{BSZ}, the authors gave a explicit formula for left Pieri rule, which could be negative, but still multiplicity free. The goal of this paper is to explore facts about these constants. In section {\ref{sec:3}}, we give a example that the immaculate analogue of saturation conjecture fails. In section \ref{sec:4}, we show that the structure constants satisfy a translation invariant property, which does not come from the commutative case. Finally in section \ref{sec:5}, we give a combinatorial interpretation of the left Pieri rule in \cite{BSZ}, which also shows that it is multiplicity free.

\section{Definitions}

\subsection{Partitions and Compositions}

A partition of a non-negative integer $n$ is an integer vector $\lambda=(\lambda_1,\lambda_2,\dots,\lambda_k)$ such that $\lambda_1\geq\lambda_2\geq\dots\geq\lambda_k>0$ and $\lambda_1+\lambda_2+\dots+\lambda_k=n$. We indicate that $\lambda$ is a partition of $n$ with the notation $\lambda\vdash n$. We denote the length $\ell(\lambda)=m$ and size $|\lambda|=n$.\\
\\
A composition of a natural number $n$ is an integer vector $\alpha=(\alpha_1,\alpha_2,\dots,\alpha_m)$ such that $\alpha_i>0$ for all $i$ and $\alpha_1+\alpha_2+\dots+\alpha_m=n$. We indicate that $\alpha$ is a composition of $n$ with the notation $\alpha\models n$. We denote the length $\ell(\alpha)=m$ and size $|\alpha|=n$.\\
\\
For partitions $\mu,\nu$, and natural number $i$, we say $\mu\prec_i\nu$ if
\begin{enumerate}
\item $|\nu|=|\mu|+i$,
\item $\mu_j\leq\nu_j$ for all $1\leq j\leq \ell(\mu)$,
\item If $\nu_i>\mu_i$ and $\nu_j>\mu_j$ for some $i, j$, then either $\nu_i\leq\mu_j$ or $\nu_j\leq\mu_i$.
\end{enumerate}
By convention, we let $\mu_j=0$ for $j>\ell(\mu)$.\\
\\
For compositions $\alpha,\beta$, and natural number $i$, we say $\alpha\subset_i\beta$ if
\begin{enumerate}
\item $|\beta|=|\alpha|+i$,
\item $\alpha_j\leq\beta_j$ for all $1\leq j\leq \ell(\alpha)$,
\item $\ell(\beta)\leq\ell(\alpha)+1$.
\end{enumerate}

\subsection{$\sym$ and $\Nsym$}

The ring of symmetric function, $\sym$ is defined as the commutative $\mathbb{Q}$-algebra generated by $\{h_1, h_2, \dots\}$ with no other relation. It is graded, namely, $h_i$ has degree $i$ and the grading is extended multiplicatively. The element $h_i$ is said to be the \textit{complete homogeneous function} of degree $i$. The sub vector space of elements of degree $n$ of $\sym$ is denoted by $\sym_n$. For a partition $\lambda=(\lambda_1,\lambda_2,\dots,\lambda_k)$, we define $h_{\lambda}=h_{\lambda_1}h_{\lambda_2}\cdots h_{\lambda_k}$. Then, $\{h_\lambda:\lambda\vdash n\}$ is a basis for $\sym_n$.\\
\\
The ring of non-commutative symmetric functions, $\Nsym$, is the $\mathbb{Q}$-algebra generated by $\{H_1,H_2,\dots\}$ with no relation. It is graded, namely, $H_i$ has degree $i$ and the grading is extended multiplicatively. The sub vector space of elements of degree $n$ of $\Nsym$ is denoted by $\Nsym_n$. For a composition $\alpha=(\alpha_1,\alpha_2,\dots,\alpha_m)$, we define $H_{\alpha}=H_{\alpha_1}H_{\alpha_2}\cdots H_{\alpha_m}$. Then, $\{H_{\alpha}:\alpha\models n\}$ is a basis for $\Nsym_n$. There is a \textit{forgetful} projection $\chi:\Nsym\to\sym$ by $\chi(H_{\alpha})=h_{\text{sort}(\alpha)}$ where sort$(\alpha)$ is the partition obtained by reordering $\alpha$.\\
\\

\subsection{The Schur functions and immaculate functions}

(For more information on Schur functions, we refer the readers to \cite{Sagan} or \cite{MacDonald})\\
\\
For a partition $\lambda=(\lambda_1,\lambda_2,\dots,\lambda_k)$, the Schur function $s_{\lambda}$ is defined as
$$s_{\lambda}=\det\begin{bmatrix}h_{\lambda_1}&h_{\lambda_1+1}&\cdots&h_{\lambda_1+k-1}\\h_{\lambda_2}-1&h_{\lambda_2}&\cdots&h_{\lambda_2+k-2}\\ \vdots&\vdots&\ddots&\vdots\\h_{\lambda_k-k+1}&h_{\lambda_k-k+2}&\cdots&h_{\lambda_k}\end{bmatrix}=\sum_{\sigma\in S_k}(-1)^{\sigma}h_{\lambda_1+\sigma_1-1,\dots,\lambda_k+\sigma_k-k}$$
where $h_0=1$ and $h_m=0$ for $m<0$.\\
\\
Then, the set $\{s_{\lambda}:\lambda\vdash n\}$ forms a basis for $\sym_n$.\\
\\
In \cite{BBSSZ}, the authors introduced a new basis for $\Nsym$ that is analogous to the Schur functions for $\sym$, called the immaculate functions, denoted by $\{\fS_{\alpha}\}$ indexed by compositions.\\
\\
For a composition $\alpha=(\alpha_1,\alpha_2,\dots,\alpha_k)$, 
\begin{equation}\label{eq:1}
\fS_{\alpha}=\sum_{\sigma\in S_k}(-1)^{\sigma}H_{\alpha_1+\sigma_1-1,\dots,\alpha_k+\sigma_k-k}
\end{equation}
where $H_0=1$ and $H_m=0$ for $m<0$.\\
\\
The set $\{\fS_{\alpha}:\alpha\models n\}$ forms a basis for $\Nsym_n$.\\
\\

\subsection{Tableaux}

A tableau is a finite collection of cells, arranged in left-justified rows and filled with positive integers e.g. 
$$ T= \tikztableausmall{{1,1,2},{2},{2,3}} $$
Let $T$ be a tableau. The shape of $T$, denoted by $sh(T)$, is the integer vector whose $i$-th entry is the length of row $i$, reading from top to bottom. The content of $T$, denoted by $c(T)$, is the integer vector whose $j$-th entry is the number of times $j$ appears in $T$. The reading word of $T$, denoted by $read(T)$, is the word of entries read starting in the top row from right to left, then proceeding down the rows. In the example above, $sh(T)=(3,1,2)$, $c(T)=(2,3,1)$, and $read(T)=211232$.\\
\\
For compositions $\alpha,\beta$ with $\alpha_i\geq\beta_i$ for all $i$, a skew tableau of shape $\alpha/\beta$ is a tableau of shape $\alpha$ with cells of $\beta$ removed from the upper left corner e.g.
$$ T'= \tikztableausmall{{\boldentry X, 1,1,2},{\boldentry X,\boldentry X,2},{2,3}} $$
has skew shape $(4,3,2)/(1,2)$.\\
\\
In this case, $\beta$ is called the \textit{inner shape} of $T'$ and $\alpha-\beta$ is called the \textit{outer shape} of $T'$.\\
\\
A tableau is called \textit{semistandard} if its rows are weakly increasing from left to right and its columns are strictly increasing from top to bottom. A tableau is said to be $immaculate$ if its rows are weakly increasing from left to right and its first column is strictly increasing from top to bottom. A tableau $T$ is called \textit{Yamanouchi} if in $read(T)$, for every positive integer $j$ and every prefix $d$ of $read(T)$, there are at least as many occurrences of $j$ as there are of $j+1$ in $d$.\\
\\
All the definitions above for tableaux apply verbatim to skew tableaux.

\subsection{The Structure Constants}

As mentioned before, Schur functions and immaculate functions form bases for $\sym$ and $\Nsym$ respectively. Then, the product of two functions can be expanded as a linear combination in their basis. More precisely, for partitions $\mu,\nu$ and compositions $\alpha,\beta$,
\[ s_{\mu}s_{\nu}=\sum_{\lambda}c_{\mu\nu}^{\lambda}s_{\lambda}, \]
\[ \fS_{\alpha}\fS_{\beta}=\sum_{\gamma}C_{\alpha\beta}^{\gamma}\fS_{\gamma}. \]
$c_{\mu\nu}^{\lambda}$ and $C_{\alpha\beta}^{\gamma}$ are called the structure constants for Schur functions and immaculate functions respectively.\\
\\
Here, we give two classical results regarding the structure constants for Schur functions. For detailed proof, we refer the readers to \cite{Sagan} and \cite{MacDonald}. 

\begin{Theorem}\label{thm:2.1}
(Littlewood-Richardson Rule) For partitions $\mu$ and $\nu$, 
$$s_{\mu}s_{\nu}=\sum_{\lambda\vdash|\mu|+|\nu|}c_{\mu\nu}^{\lambda}s_{\lambda}$$
where $c_{\mu\nu}^{\lambda}$ is the number of skew semistandard Yamanouchi tableaux of shape $\lambda/\mu$ and content $\nu$.
\end{Theorem}

\noindent
The Littlewood-Richardson Rule gives an explicit combinatorial interpretation to the structure constants. And in particular, when $\nu=n$ is an integer, $s_{\nu}=h_n$ and we have the following Pieri Rule.

\begin{Theorem}
(Pieri Rule) For partition $\mu$ and natural number $n$, 
$$s_{\mu}h_{n}=\sum_{\nu\prec_n \lambda}s_{\lambda}.$$
\end{Theorem}

\noindent
In the non-commutative case, we have the following two theorems that are analogous to what we had above.
\begin{Theorem}\label{thm:2.3}
(\cite{BBSSZ2} Theorem 3.5) For composition $\alpha$ and natural number $s$,
$$\fS_{\alpha}H_s=\sum_{\alpha\subset_s\beta}\fS_{\beta}.$$
\end{Theorem}

\begin{Theorem}\label{thm:2.4}
(\cite{BBSSZ} Theorem 7.3) For composition $\alpha$ and partition $\lambda$,
$$\fS_{\alpha}\fS_{\lambda}=\sum_{\gamma\models|\alpha|+|\lambda|}C_{\alpha,\lambda}^{\gamma}\fS_{\gamma}$$
where $C_{\alpha,\lambda}^{\gamma}$ is the number of skew immaculate Yamanouchi tableaux of shape $\gamma/\alpha$ and content $\lambda$.
\end{Theorem}

\noindent
Unfortunately, we do not have a nice formula for general structure constants because they could be negative, and no general cancellation-free formula is known for a composition $\beta$ instead of a partition $\lambda$.

\begin{Example}\label{ex:2.5}
\[ \fS_2\fS_{2,4}=\fS_{3,1,4}+\fS_{2,2,4}+\fS_{3,2,3}-\fS_{5,3}-\fS_{4,3,1}. \]
\end{Example}

\noindent
In \cite{BSZ}, the authors gave a formula for Left-Pieri Rule using the dual Hopf algebra, quasi-symmetric functions ($\Qsym$). In section 5, we will give another proof of that using a combinatorial approach.\\
\\

\section{The Saturation Conjecture}\label{sec:3}

\noindent
Let $T$ denote the set of triples $\{(\lambda,\mu,\nu)\mid c_{\mu\nu}^{\lambda}\neq0\}$. This is an important set that also appears in many other branches of mathematics such as representation of $GL_n(\mathbb{C})$ and Schubert calculus. There is an old conjecture, proved by A.Knutson and T.Tao \cite{KT}, asserts that $T_n$ is saturated in $\mathbb{Z}^{3n}$.\\
\\
For a composition $\alpha=(\alpha_1,\dots,\alpha_m)$ and a positive integer $N$, we denote their pointwise product $(N\alpha_1,\dots,N\alpha_m)$ by $N\alpha$.

\begin{Theorem}
 (Saturation Conjecture) For partitions $\mu,\nu$ and $\lambda$, and positive integer $N$, $c_{\mu\nu}^{\lambda}\neq0$ if and only if $c_{N\mu,N\nu}^{N\lambda}\neq0$.
\end{Theorem}

\noindent
Naturally, a non-commutative analogue for saturation conjecture for compositions $\alpha,\beta$ and $\gamma$, and positive integer $N$ is: $C_{\alpha\beta}^{\gamma}\neq0$ if and only if $C_{N\alpha,N\beta}^{N\gamma}\neq0$. However, in this case, the conjecture fails and we provide the following counter example.\\
\\
Consider the case $\alpha=(1,1)$, $\beta=(3,2,2)$, $\gamma=(3,3,1,1,1)$, and $N=2$. Since $\beta$ is a partition, by Theorem \ref{thm:2.4}, $C_{\alpha\beta}^{\gamma}\neq0$ if and only if there exists a skew immaculate Yamanouchi tableau of shape $\gamma/\alpha$ and content $\beta$, and $C_{N\alpha,N\beta}^{N\gamma}\neq0$ if and only if there exists a skew immaculate Yamanouchi tableau of shape $N\gamma/N\alpha$ and content $N\beta$.\\
\\
Now, the tableau
\[ \tikztableausmall{{\boldentry X, \boldentry X, 1,1,1,1},{\boldentry X,\boldentry X,1,2,2,2},{1,3},{2,3},{3,3}} \]
is skew immaculate Yamanouchi, and hence $C_{N\alpha,N\beta}^{N\gamma}\neq0$. In fact, $C_{N\alpha,N\beta}^{N\gamma}=1$.\\
\\
On the other hand, for a tableau of shape $\gamma/\alpha$ and content $\beta$,
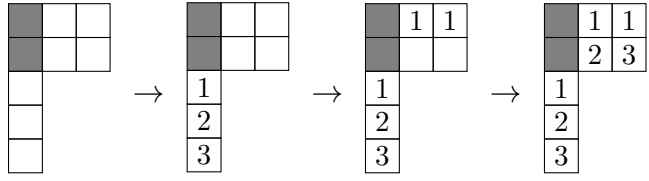
\[ \tikztableausmall{{\boldentry X, , },{\boldentry X, , },{ },{ },{ }} \to  \tikztableausmall{{\boldentry X, , },{\boldentry X, , },{1},{2},{3}} \to  \tikztableausmall{{\boldentry X,1,1},{\boldentry X, , },{1},{2},{3}} \to  \tikztableausmall{{\boldentry X,1,1},{\boldentry X,2,3},{1},{2},{3}} \]
cannot be immaculate Yamanouchi. To be immaculate, the first column must be strictly increasing i.e. $(1,2,3)$. To be Yamanouchi, the first row must consist of $1$ only i.e. $(1,1)$. Then, the second row has to be $(2,3)$, as it is immaculate. But the resulting tableau fails to be Yamanouchi, as $read(T)=(1,1,3,2,1,2,3)$ fails. Therefore, $C_{\alpha\beta}^{\gamma}=0$.\\
\\

\section{Translation Invariance}\label{sec:4}

\noindent
In this section, we give a useful fact that helps understand the structure constants for immaculate functions. In \cite{BBSSZ}, the authors proved this property when $\beta$ is a partition by using Theorem \ref{thm:2.4}. However, a general formula was not given there. We give here a simpler, fully general proof.

\begin{Theorem}\label{thm:4.1}
For compositions $\alpha,\beta,\gamma,v$ with $\ell(v)\leq\ell(\alpha)$, we have $C_{\alpha,\beta}^{\gamma}=C_{\alpha+v,\beta}^{\gamma+v}.$
\end{Theorem}

\begin{proof}

Let $\beta=(\beta_1,\dots,\beta_m)$. Using the definition of $\fS_{\beta}$ (\ref{eq:1}), we have

\[ \fS_{\alpha}\fS_{\beta}=\sum_{\sigma\in S_m}(-1)^{\sigma}\fS_{\alpha}H_{\beta_1+\sigma_1-1,\beta_2+\sigma_2-2,\dots,\beta_m+\sigma_m-m}. \]

\noindent
An iterative use of the right Pieri rule (Theorem \ref{thm:2.3}) gives

\[ \fS_{\alpha}\HH_{\tau}=\sum_{sh(T)=\gamma/\alpha \atop c(T)=\tau}\fS_{\gamma} \]

\noindent
where $\tau$ is a composition and the sum is over all skew immaculate tableau $T$.\\
\\
Combining the two equations above yields

\[ \fS_{\alpha}\fS_{\beta}=\sum_{\sigma\in S_m}\sum_{sh(T)=\gamma/\alpha \atop c(T)=\beta+\sigma-Id}(-1)^{\sigma}\fS_{\gamma}. \]

\noindent
Let $\mathfrak{T}_{\alpha}^{\beta}$ be the set of skew immaculate tableaux of inner shape $\alpha$ for which $c(T)-\beta+Id$ is a permutation (in one-line notation), where $Id=(1,2,.\dots,m)$. In this case, entries in $T$ must be in $\{1,2,\dots,m\}$ and $c(T)$ means the content vector of length $m$ by padding $0$'s. Let $\sigma(T)=c(T)-\beta+Id$, we have $(16)$ in \cite{BBSSZ}:

\begin{equation}\label{eq:2}
\fS_{\alpha}\fS_{\beta}=\sum_{T\in\mathfrak{T}_{\alpha}^{\beta}}(-1)^{\sigma(T)}\fS_{sh(T)}
\end{equation}

\begin{Example}
Let $\alpha=(1)$, $\beta=(1,3,1)$, then
\[ T=\tikztableausmall{{\boldentry X,2},{1,2,2},{2}} \in\mathfrak{T}_{\alpha}^{\beta} \]
\end{Example}

\noindent
and $\sigma(T)=(1,4,0)-(1,3,1)+(1,2,3)=(1,3,2)$.\\
\\
For each $T\in\mathfrak{T}_{\alpha}^{\beta}$ with $sh(T)=\gamma/\alpha$, we move each of the first $\ell(v)$ rows of $T$ to the
right by a certain number of steps, namely, the $i$-th row by $v_i$ steps, where $v=(v_1,v_2,\dots,v_{\ell(v)})$. By this construction, we obtain a tableau $T'\in\mathfrak{T}_{\alpha+v}^{\beta}$ with $sh(T')=(\gamma+v)/\alpha$, and vice versa. Since $\ell(v)\leq\ell(\alpha)$, the first column is preserved under this map. Moreover, $c(T)=c(T')$ and hence $\sigma(T)=\sigma(T')$. Then, the result follows.
\end{proof}

\noindent
Therefore, in order to understand $C_{\alpha\beta}^{\gamma}$, it suffices to understand those when $\alpha$ is the $n$-tuple $(1,1,\dots,1)$ for $n\in\mathbb{N}$.\\
\\
\begin{Example}
Let $\alpha=(1,1,2)$, $\beta=(2,1,3)$, $v=(1,2)$, then
\[ \tikztableausmall{{\boldentry X,1},{\boldentry X,1,2,3},{\boldentry X,\boldentry X,3,3}} \in\mathfrak{T}_{\alpha}^{\beta}\Leftrightarrow \tikztableausmall{{\boldentry X,\boldentry X,1},{\boldentry X,\boldentry X,\boldentry X,1,2,3},{\boldentry X,\boldentry X,3,3}} \in\mathfrak{T}_{\alpha+v}^{\beta}. \]
\end{Example}

\section{Left Pieri Rule}\label{sec:5}

\noindent
Unlike the commutative case, for immaculate functions, the left Pieri rule is much different from the right Pieri rule. As shown in Example \ref{ex:2.5}, the structure constants could be negative.\\
\\
Theorem \ref{thm:4.1} tells that formulating the left Pieri rule is equivalent to understanding $H_1\fS_{\beta}=\fS_1\fS_{\beta}$. Equation (\ref{eq:2}) gives a combinatorial interpretation of the coefficients, but with a sign. In \cite{BBSSZ}, the authors proved Theorem \ref{thm:2.4} by using a sign-reversing involution. Inspired by that, we modify that involution and obtain a cancellation-free formula for the coefficients.\\
\\
Let $\mathfrak{T}_{\alpha}^{\beta}$ be as defined in section ~\ref{sec:4} with $\ell(\beta)=n$. First, we define an involution from $\mathfrak{T}_{\alpha}^{\beta}$ to itself.

\begin{Definition}
For each tableau $T\in\mathfrak{T}_{\alpha}^{\beta}$, we construct a tableau $y(T)$ as follows. For every cell of content $r$ in the $i$-th row of $T$, we put a cell of content $i$ in the $\sigma(T)(r)$-th row of $y(T)$. We sort the entries of the rows of $y(T)$ in non-decreasing order. In general, $y(T)$ is to be a straight-shape tableau, and might have empty rows.\\
\\
We define a function $Y$ that maps $T\in\mathfrak{T}_{\alpha}^{\beta}$ to the pair $(y(T),\sigma(T))$. Note that $Y$ is injective i.e. fixing $\alpha$, we can recover $T$ from $(y(T),\sigma(T))$. More precisely, for every cell of content $r$ in the $i$-th row of $T'$, we put a cell of content $\sigma(T)^{-1}(i)$ in the $r$-th row of $T$.\\
\\
We define $Y^{-1}$ to be the reversed construction from a pair $(T',\sigma)$ to $T$ where $T'$ is a tableau with entries $\{1,\dots,n\}$ and $\sigma$ is a permutation in $S_n$. Here, $Y^{-1}$ is not the inverse map of $Y$ because $Y^{-1}(T,\sigma)$ may not be immaculate i.e. the domain of $Y$ is not equal to the image of $Y^{-1}$.
\end{Definition}

\noindent
In this case, $Y^{-1}\circ Y$ is the identity map while $Y\circ Y^{-1}$ is not because $Y^{-1}$ has a much larger domain.

\begin{Definition}
We say a cell $x$ not in the first row  with value $a$ is \textit{nefarious} if the cell above $x$ is either empty or it contains $b$ with $b\geq a$ i.e.
$$ \tikztableausmall{{\boldentry X},{$a$}} \text{   or   } \tikztableausmall{{$b$},{$a$}} $$
\end{Definition}

\begin{Example}
Let $\alpha=(1,2)$ and $\beta=(2,2,2)$. Let
$$ T= \tikztableausmall{{\boldentry X, 1,1,2},{\boldentry X,\boldentry X,2},{2,3}} $$
Note that $\sigma(T)=c(T)-\beta+Id=(2,3,1)-(2,2,2)+(1,2,3)=(1,3,2)$, hence,
$$ y(T)= \tikztableausmall{{1,1},{3},{1,2,3}} $$
and the nefarious cells in $y(T)$ are the three cells in the third row.
\end{Example}

The next definition defines a key involution that is a modified version of the Lindstrom-Gessel-Viennot swap, which is usually illustrated on lattice path, but can be applied to tableaux equally. The equivalence can be found in chapter $4.5$ of \cite{Sagan}.

\begin{Definition}
For each $(y(T),\sigma(T))\in Y(\mathfrak{T}_{\alpha}^{\beta})$ that contains a nefarious cell $x$, we define a tableau $\theta_x(y(T))$ and a pair $\Theta_x(y(T),\sigma(T))$ as follows:\\
\\
Let the cell $x$ appear in the $(r+1)$-th row of $y(T)$.
\begin{enumerate}
\item If the cell $y$ above $x$ is not empty, then define $\theta_x(y(T))$ to be the tableau obtained from $y(T)$ by moving:
\begin{enumerate}
\item all the cells strictly to the right of $x$ into the row above $x$
\item all the cells weakly to the right of $y$ into row $r+1$.
\end{enumerate}
\item Otherwise, move all the cells strictly to the right of $x$ into row $r$.
\end{enumerate}

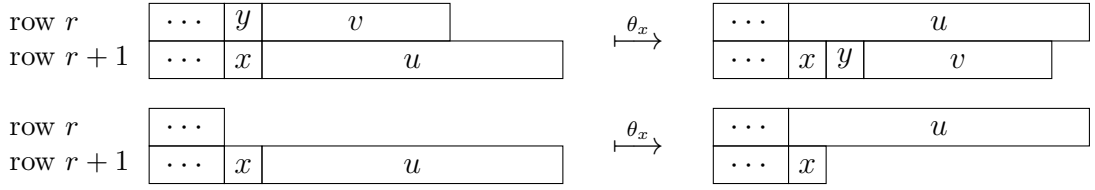
\begin{figure}[h!]
\begin{center}
\begin{tikzpicture}[scale=0.5, every node/.style={anchor=south}]
  \coordinate (o) at (0, 0);
  \node[anchor=south west] at ($(o) +(-6,0)$) {\small row $r+1$};
  \node[anchor=south west] at ($(o) +(-6,1)$) {\small row $r$};
  \draw[color=black] (o) rectangle +(-2, 1);
  \node at ($(o) +(-1,0)$) {$\cdots$};
  \draw[color=black] ($(o) +(0,1)$) rectangle +(-2, 1);
  \node at ($(o) +(-1,1)$) {$\cdots$};
  \draw[color=black] (o) rectangle +(1, 1);
  \node at ($(o) +(0.5,0)$) {$x$};
  \draw[color=black] ($(o) + (0, 1)$) rectangle +(1, 1);
  \node[color=black] at ($(o) +(0.5,1)$) {$y$};
  \draw[color=black] ($(o) + (1, 0)$) rectangle +(8, 1);
  \node[color=black] at ($(o) +(5,0)$) {$u$};
  \draw[color=black] ($(o) + (1, 1)$) rectangle +(5, 1);
  \node[color=black] at ($(o) +(3.5,1)$) {$v$};
  \node at ($(o) +(11,0.5)$) {${\buildrel \theta_x\over\longmapsto}$};
  \coordinate (o) at (15, 0);
  \draw[color=black] (o) rectangle +(-2, 1);
  \node at ($(o) +(-1,0)$) {$\cdots$};
  \draw[color=black] ($(o) +(0,1)$) rectangle +(-2, 1);
  \node at ($(o) +(-1,1)$) {$\cdots$};
  \draw[color=black] (o) rectangle +(1, 1);
  \node at ($(o) +(0.5,0)$) {$x$};
  \draw[color=black] ($(o) + (1, 0)$) rectangle +(1, 1);
  \node[color=black] at ($(o) +(1.5,0)$) {$y$};
  \draw[color=black] ($(o) + (0, 1)$) rectangle +(8, 1);
  \node[color=black] at ($(o) +(4,1)$) {$u$};
  \draw[color=black] ($(o) + (2, 0)$) rectangle +(5, 1);
  \node[color=black] at ($(o) +(4.5,0)$) {$v$};
\end{tikzpicture}

\medskip
\medskip
 
\begin{tikzpicture}[scale=0.5, every node/.style={anchor=south}]
  \coordinate (o) at (0, 0);
  \node[anchor=south west] at ($(o) +(-6,0)$) {\small row $r+1$};
  \node[anchor=south west] at ($(o) +(-6,1)$) {\small row $r$};
  \draw[color=black] (o) rectangle +(-2, 1);
  \node at ($(o) +(-1,0)$) {$\cdots$};
  \draw[color=black] ($(o) +(0,1)$) rectangle +(-2, 1);
  \node at ($(o) +(-1,1)$) {$\cdots$};
  \draw[color=black] (o) rectangle +(1, 1);
  \node at ($(o) +(0.5,0)$) {$x$};
  \draw[color=black] ($(o) + (1, 0)$) rectangle +(8, 1);
  \node[color=black] at ($(o) +(5,0)$) {$u$};
  \node at ($(o) +(11,0.5)$) {${\buildrel \theta_x\over\longmapsto}$};
  \coordinate (o) at (15, 0);
  \draw[color=black] (o) rectangle +(-2, 1);
  \node at ($(o) +(-1,0)$) {$\cdots$};
  \draw[color=black] ($(o) +(0,1)$) rectangle +(-2, 1);
  \node at ($(o) +(-1,1)$) {$\cdots$};
  \draw[color=black] (o) rectangle +(1, 1);
  \node at ($(o) +(0.5,0)$) {$x$};
  \draw[color=black] ($(o) + (0, 1)$) rectangle +(8, 1);
  \node[color=black] at ($(o) +(4,1)$) {$u$};
\end{tikzpicture}
\end{center}

\label{fig:Theta2}
\end{figure}

\noindent
Let $t_r=(1,2,
\dots,r-1,r+1,r,r+2,r+3,
\dots,n)$ be the transposition of $r$ and $r+1$. Then, $
\Theta_x$ maps the pair $(y(T),\sigma(T))$ to $(\theta_x(y(T)),t_r\circ\sigma(T))$.
\end{Definition}

\begin{Example}
Let $x$ be the second cell in row 2. Then, $\theta_x$ maps
$$ \tikztableausmall{{$1$,$2$,$3$},{$2$,$2$,$3$}} \to \tikztableausmall{{$1$,$3$},{$2$,$2$,$2$,$3$}}$$
or
$$ \tikztableausmall{{$1$},{$2$,$2$,$3$}} \to \tikztableausmall{{$1$,$3$},{$2$,$2$}}. $$
\end{Example}

\begin{Definition}
For $1\leq r\leq\ell(\alpha)$ and a cell $x$ in row $r$ of $y(T)$, we say that $x$ is \textit{the most nefarious cell in row} $r$ if it is the left-most nefarious cell in row $r$ such that $Y^{-1}\circ\Theta_x\circ Y$ fixes the first column of $T$. In particular, $Y^{-1}\circ\Theta_x\circ Y$ is immaculate.\\
\\
Then, we can define a map $\Phi_r:\mathfrak{T}_{\alpha}^{\beta}\to\mathfrak{T}_{\alpha}^{\beta}$ by either $\Phi_r(T)=Y^{-1}\circ\Theta_x\circ Y(T)$ where $x$ is the most nefarious cell in row $r$ of $y(T)$, or $T$ is fixed by $\Phi_r$ if there is no most nefarious cell in row $r$ of $y(T)$.
\end{Definition}

\noindent
For every $r$, $\Phi_r$ has the following properties.

\begin{Lemma}
For each $T\in\mathfrak{T}_{\alpha}^{\beta}$,
\begin{enumerate}
\item If there exists a most nefarious cell $x$ in row $r$, then it must be the left-most nefarious cell in row $r$.
\item $\Phi_r$ is an involution i.e. $\Phi_r^2=id$.
\item $T$ and $\Phi_r(T)$ have the same shape.
\item If $T$ is not fixed by $\Phi_r$, then $\sigma(T)$ and $\sigma(\Phi_r(T))$ have opposite sign.
\end{enumerate}
\end{Lemma}

\begin{proof}
\noindent
\begin{figure}[h!]
\begin{center}
\begin{tikzpicture}[scale=0.5, every node/.style={anchor=south}]
  \coordinate (o) at (0, 0);
  \node[anchor=south west] at ($(o) +(-6.4,-0.2)$) {\small row $c(a_i)$};
  \node[anchor=south west] at ($(o) +(-6.4,0.8)$) {\small row $c(b_i)$};
  \node[anchor=south west] at ($(o) +(-6.4,-1.2)$) {\small row $c(x)$};
  \draw[color=black] ($(o) +(-3,-1)$) rectangle +(3.5, 1);
  \node at ($(o) +(-1,-1)$) {$\cdots$};
  \draw[color=black] ($(o) + (0.5, -1)$) rectangle +(5, 1);
  \node at ($(o) +(3,-1.25)$) {$\sigma^{-1}(r)$};
  \draw[color=black] ($(o) + (5.5, -1)$) rectangle +(2, 1);
  \node[color=black] at ($(o) +(6.5,-1)$) {$\cdots$};
  \draw[color=black] ($(o) +(-3,0)$) rectangle +(2, 1);
  \node at ($(o) +(-2,0)$) {$\cdots$};
  \draw[color=black] ($(o) +(-3,1)$) rectangle +(3, 1);
  \node at ($(o) +(-1.5,1)$) {$\cdots$};
  \draw[color=black] ($(o) + (-1, 0)$) rectangle +(5, 1);
  \node at ($(o) +(1.5,-0.25)$) {$\sigma^{-1}(r)$};
  \draw[color=black] ($(o) + (0, 1)$) rectangle +(5, 1);
  \node[color=black] at ($(o) +(2.5,0.75)$) {$\sigma^{-1}(r-1)$};
  \draw[color=black] ($(o) + (4, 0)$) rectangle +(2, 1);
  \node[color=black] at ($(o) +(5,0)$) {$\cdots$};
  \draw[color=black] ($(o) + (5, 1)$) rectangle +(2, 1);
  \node[color=black] at ($(o) +(6,1)$) {$\cdots$};
  \node at ($(o) +(9,0.5)$) {${\buildrel y\over\longmapsto}$};
  \coordinate (o) at (16, 0);
  \node[anchor=south west] at ($(o) +(-6,0)$) {\small row $r$};
  \node[anchor=south west] at ($(o) +(-6,1)$) {\small row $r-1$};
  \draw[color=black] ($(o) +(-2,0)$) rectangle +(1, 1);
  \node at ($(o) +(-1.4,-0.2)$) {$a_1$};
  \draw[color=black] ($(o) +(-2,1)$) rectangle +(1, 1);
  \node at ($(o) +(-1.4,0.8)$) {$b_1$};
  \draw[color=black] ($(o) +(-1,0)$) rectangle +(2, 1);
  \node at ($(o) +(0,0)$) {$\cdots$};
  \draw[color=black] ($(o) +(-1,1)$) rectangle +(2, 1);
  \node at ($(o) +(0,1)$) {$\cdots$};
  \draw[color=black] ($(o) +(1,0)$) rectangle +(1, 1);
  \node at ($(o) +(1.6,-0.2)$) {$a_k$};
  \draw[color=black] ($(o) +(1,1)$) rectangle +(1, 1);
  \node at ($(o) +(1.6,0.8)$) {$b_k$};
  \draw[color=black] ($(o) +(2,0)$) rectangle +(1, 1);
  \node at ($(o) +(2.5,0)$) {$x$};
  \draw[color=black] ($(o) + (2, 1)$) rectangle +(2, 1);
  \node[color=black] at ($(o) +(3,1)$) {$u$};
  \draw[color=black] ($(o) + (3, 0)$) rectangle +(4, 1);
  \node[color=black] at ($(o) +(5,0)$) {$v$};
\end{tikzpicture}

\medskip
\medskip

\begin{tikzpicture}[scale=0.5, every node/.style={anchor=south}]
  \coordinate (o) at (2, 0);
  \node at ($(o) +(-7,0.5)$) {${\buildrel \theta_x\over\longmapsto}$};
  \node[anchor=south west] at ($(o) +(-6,0)$) {\small row $r$};
  \node[anchor=south west] at ($(o) +(-6,1)$) {\small row $r-1$};
  \draw[color=black] ($(o) +(-2,0)$) rectangle +(1, 1);
  \node at ($(o) +(-1.4,-0.2)$) {$a_1$};
  \draw[color=black] ($(o) +(-2,1)$) rectangle +(1, 1);
  \node at ($(o) +(-1.4,0.8)$) {$b_1$};
  \draw[color=black] ($(o) +(-1,0)$) rectangle +(2, 1);
  \node at ($(o) +(0,0)$) {$\cdots$};
  \draw[color=black] ($(o) +(-1,1)$) rectangle +(2, 1);
  \node at ($(o) +(0,1)$) {$\cdots$};
  \draw[color=black] ($(o) +(1,0)$) rectangle +(1, 1);
  \node at ($(o) +(1.6,-0.2)$) {$a_k$};
  \draw[color=black] ($(o) +(1,1)$) rectangle +(1, 1);
  \node at ($(o) +(1.6,0.8)$) {$b_k$};
  \draw[color=black] ($(o) +(2,0)$) rectangle +(1, 1);
  \node at ($(o) +(2.5,0)$) {$x$};
  \draw[color=black] ($(o) + (2, 1)$) rectangle +(4, 1);
  \node[color=black] at ($(o) +(4,1)$) {$v$};
  \draw[color=black] ($(o) + (3, 0)$) rectangle +(2, 1);
  \node[color=black] at ($(o) +(4,0)$) {$u$};
  \node at ($(o) +(7,0.5)$) {${\buildrel y^{-1}\over\longmapsto}$};
  \coordinate (o) at (16, 0);
  \node[anchor=south west] at ($(o) +(-6.4,-0.2)$) {\small row $c(a_i)$};
  \node[anchor=south west] at ($(o) +(-6.4,0.8)$) {\small row $c(b_i)$};
  \node[anchor=south west] at ($(o) +(-6.4,-1.2)$) {\small row $c(x)$};
  \draw[color=black] ($(o) +(-3,-1)$) rectangle +(3.5, 1);
  \node at ($(o) +(-1,-1)$) {$\cdots$};
  \draw[color=black] ($(o) + (0.5, -1)$) rectangle +(5, 1);
  \node at ($(o) +(3,-1.25)$) {$\sigma^{-1}(r-1)$};
  \draw[color=black] ($(o) + (5.5, -1)$) rectangle +(2, 1);
  \node[color=black] at ($(o) +(6.5,-1)$) {$\cdots$};
  \draw[color=black] ($(o) +(-3,0)$) rectangle +(2, 1);
  \node at ($(o) +(-2,0)$) {$\cdots$};
  \draw[color=black] ($(o) +(-3,1)$) rectangle +(3, 1);
  \node at ($(o) +(-1.5,1)$) {$\cdots$};
  \draw[color=black] ($(o) + (-1, 0)$) rectangle +(5, 1);
  \node at ($(o) +(1.5,-0.25)$) {$\sigma^{-1}(r-1)$};
  \draw[color=black] ($(o) + (0, 1)$) rectangle +(5, 1);
  \node[color=black] at ($(o) +(2.5,0.75)$) {$\sigma^{-1}(r)$};
  \draw[color=black] ($(o) + (4, 0)$) rectangle +(2, 1);
  \node[color=black] at ($(o) +(5,0)$) {$\cdots$};
  \draw[color=black] ($(o) + (5, 1)$) rectangle +(2, 1);
  \node[color=black] at ($(o) +(6,1)$) {$\cdots$};
\end{tikzpicture}

\end{center}
\end{figure}
For simplicity, we denote $\sigma(T)$ by $\sigma$.

(1) Suppose $x$ is the most nefarious cell in row $r$ of $y(T)$, as shown in the figure above, with entry $c(x)$. Let $c(y)$ be the entry in cell $y$ in $y(T)$. To obtain $\Phi_r(T)$ from $T$, it suffices to do the following. For every cell $y$ in row $r$ of $y (T)$ which lies weakly to the left of $x$, we replace a $\sigma^{-1}(r)$ in row $c(y)$
of $T$ by a $\sigma^{-1}(r-1)$. For every cell $y$ in row $r-1$ of $y(T)$ which lies strictly to the left of $x$, we
replace a $\sigma^{-1}(r-1)$ in row $c(y)$ of $T$ by a $\sigma^{-1}(r)$. Since $Y^{-1}\circ\Theta_x\circ Y$ fixes the first column of $T$, and if a cell $a_i$ is nefarious, then $Y^{-1}\circ\Theta_{a_i}\circ Y$ also fixes the first column of $T$, because we interchange less cells. Therefore, if $x$ is the most nefarious cell, it must be the left-most nefarious cell.

(2) If there is no most nefarious cell in row $r$, then $\Phi_r(T)=T$ and $\Phi_r^2(T)=T$. Otherwise, since $\Phi_r$ preserves the first column of $T$, $Y^{-1}(\theta_x\circ y(T),t_r\circ\sigma(T))\in\mathfrak{T}_{\alpha}^{\beta}$ and hence, $Y^{-1}\circ Y=id$. By part (1), the most nefarious cell remains unchanged under $\theta_x$. Therefore, $\Phi_r^2(T)=T$ as $\Theta_x^2(T)=T$, and $\Phi_r$ is an involution.

(3) $sh(T)=c(y(T))$ and $\Theta_x$ preserves the content of $y(T)$.

(4) By the definition of $\Theta_x$, if $\Phi_r(T)\neq T$, then $\sigma(\Phi_r(T))=t_{r-1}\circ\sigma(T)$.
\end{proof}

\noindent
Before we continue, consider the case where $\alpha=0$. Since $\fS_0=1$, we know that $\fS_0\fS_{\beta}=\fS_{\beta}$. On the other hand, we can also express $\fS_0\fS_{\beta}$ using (\ref{eq:2}). Hence, there must exist a sign-reversing involution on $\mathfrak{T}_{0}^{\beta}$ that cancels everything except the tableau corresponding to $\fS_{\beta}$, namely the unique immaculate tableau of shape $\beta$ and content $\beta$. For simplicity, we call this involution $\Phi_0$.
\\
\\
Now, we characterize the tableaux that are fixed by all $\Phi_r$. For simplicity, for $\alpha=(1)$, a composition $\beta$ and $T\in\mathfrak{T}_{\alpha}^{\beta}$, we define $\delta(T)$ as $(s+1,\delta_1,\dots,\delta_n)$ where $n=\ell(\beta)$, $s$ is the number of non-empty cells in the first row of $T$, and $\delta_i$ is the length of row that starts with $i$, not including the first row as it starts with empty cell. Here, $\delta(T)$ is an integer vector, it may not be a composition as some $\delta_i$ could be $0$. Then, $sh(T)=\gamma/\alpha=comp(\delta(T))/\alpha$ where $comp(\delta(T))$ is the composition obtained from $\delta(T)$ by removing the zeroes.

\begin{Lemma}\label{lem:5.8}
Fix $\alpha=(1)$. Let $T\in\mathfrak{T}_{\alpha}^{\beta}$ with outer shape $\gamma$ be fixed by all $\Phi_r$ and $\delta(T)$ be defined as above, then
\begin{enumerate}
\item All entries in the first row of $T$ must be the same, say $k$, and $\sigma(k)=1$. In particular, all $1$'s in $y(T)$ appear in its first row.
\item  If $\beta_1>s$, then all entries in the first row of $T$ are 1, and $\delta_1\leq\beta_1$.
\item If $\beta_1<s$, then $\sigma(1)=2$ and $\delta_1>\beta_1$.
\item If $\beta_1=s$, then
\begin{enumerate}
\item $\sigma(1)=1$ and $\delta_i=\beta_i$ for all $i>1$, or
\item $\sigma(1)=2$ and $\delta_1>\beta_1$.
\end{enumerate}
\item If $\delta_1>\beta_1$, and if $k$ is an entry in the second row of $T$ and $k\neq1$, then $k$ appears in the first row of $T$.
\end{enumerate}
\end{Lemma}

\begin{proof}

(1) Let $k_1,\dots,k_m$ be the $m$ distinct entries in the first row of $T$. Let $r=\max\{\sigma(k_i)\}$ and suppose $r>1$. Then, the first cell in row $r$ of $y(T)$ is $1$, which must be the most nefarious cell: It is clearly the left-most nefarious cell, and (if we denote it by $x$) the map $Y^{-1}\circ \Theta_x \circ Y$ fixes the first column of $T$ (since it only changes a single entry in the first row
of $T$, but the first row of $T$ does not intersect the first column).\\
\\
Applying $\Phi_{r}$ gives an involution that cancels it, because in $\Phi_r(T)$, there is a $\sigma^{-1}(r-1)$ in row $1$, and that cell again corresponds to a most nefarious cell and $r$ is still the new $\max\{t_{r-1}\circ\sigma(k_i)\}$ for $\Phi_r(T)$. Therefore, $\Phi_r$ is indeed a involution, $T\neq\Phi_r(T)$, they have opposite sign and get canceled by $\Phi_r$. This contradict to our choice of $T$.\\
\\
(2) If $\beta_1>s$ and the entries in the first row of $T$ is $k\neq 1$, then all $1$'$s$ must be in the second row of $T$, because $T$ is immaculate. We claim that $\sigma(1)=2$. If not, $\sigma(1)>2$, and the first $2$ in row $\sigma(1)$ of $y(T)$ is the most nefarious cell. The involution $\Phi_{\sigma(1)}$ fixes the first column because there are at least two $1$'s in the second row of $T$, but only one of them is changed to $\sigma^{-1}(\sigma(1)-1)$.\\
\\
That means in $y(T)$, there are at least $\beta_1+1\geq s+2$ many $2$'s in row $2$. Since there are only $s$ $1$'s appearing in $y(T)$, the $s+1$ cell, counting from the left, in row $2$ must be the nefarious cell. Applying $\Phi_{2}$ gives an involution that maps it to the following situation, that is, we have $(t_1\circ\sigma)(1)=1$, all $s$ entries in the first row are changed $1$, and $s+1$ $1$'s in the second row are changed to $\sigma^{-1}(1)$. Therefore, the first column of $T$ remain unchanged.\\
\\
In this case, we have $\sigma(1)=1$, all entries in the first row are $1$ and the remaining $1$'s are in the second row. Now, we can consider the $1$'s as empty cells, and we are in a similar situation where $\alpha'=\beta_1-s$, $\beta'=(\beta_2,\dots,\beta_n)$ and $s'=\delta_1-\beta_1+s$. By part (1), if it is fixed by all $\Phi_r$, all entries in the second row must be the same $k$ and $\sigma(k)=2$. That means there are $s'$ $2$'s in the second row of $y(T)$. But there are $s$ $1$'s in the first row of $y(T)$. If $\delta_1>\beta_1$, then $s'>s$ and the $s+1$ cell in the second row of $y(T)$ becomes the most nefarious cell. The first column of $T$ is fixed under $\Phi_2$ because the smallest entry in the second row is always $1$. Therefore, the only tableaux that are fixed by $\Phi_2$ are those as defined in the statement.\\
\\
(3) If $\beta_1<s$, then it is not possible to fill the first row with $1$ while keeping $\sigma(1)=1$ since $c(T)=\sigma(T)+\beta-Id$. Therefore, all the $1$'s in $T$ must appear in the second row. Using the same argument as in the proof of part 2, we must have $\sigma(1)=2$  Moreover, since there are $\beta_1+1$ $1$'s in the second row of $T$, we must have $\delta_1>\beta_1$.\\
\\
(4) If $\beta_1=s$, then there are two cases. If $\sigma(1)=1$, then all $1$'s appear in the first row of $T$ and we are in the case that $\alpha'=0$ and $\beta'=(\beta_2,\dots,\beta_n)$. Applying $\Phi_0$ gives the desired result.\\
\\
If $\sigma(1)>1$, then we are in the same case as (3). Hence, $\sigma(1)=2$ and $\delta_1>\beta_1$.\\
\\
(5) If $\delta_1>\beta_1$, by (2),(3) and (4), we have $\sigma(1)=2$ and all $1$'s in $T$ appear in its second row. Therefore, if there is some $k$ in the second row of $T$ that $k\neq 1$ and $k$ does not appear in the first row of $T$, we must have $\sigma(k)>2$. In this case, let $k$ to be the one with maximal $\sigma(k)$ among all entries $j$ in the second row. It corresponds to a $2$ in row $\sigma(k)$ of $y(T)$. This must be the most nefarious cell because by (1), all $1$'s in $y(T)$ are in the first row of $y(T)$. $\Phi_{\sigma(k)}$ also fixes the first column of $T$ because the first entry in the second row of $T$ remains $1$.

\end{proof}

\noindent
We now use Lemma \ref{lem:5.8} iteratively.  After determining the first row, if we are in situation (2) of the Lemma we can consider all the $1$'s in $T$ as empty cells. If we are in situation (3), we can consider $1$'s as empty cells and then remove the first row, because the first row is completely determined by the entries in the second row.  By Theorem \ref{thm:4.1}, the number of empty cells in the first row does not matter. Hence, we are back in the same situation with a different $\alpha'$, $\beta'$ and $\gamma'$, where $\alpha'$ may not be $(1)$. Therefore, we have the following more general properties.

\begin{Corollary}\label{cor:5.9}
If $\alpha=(1)$, let $T\in\mathfrak{T}_{\alpha}^{\beta}$ with outer shape $\gamma$ be fixed by all $\Phi_r$, and $\delta$ be defined as above, then,
\begin{enumerate}
\item if $\beta_i<s+\displaystyle\sum_{j=1}^{i-1}(\delta_j-\beta_j)$, then $\beta_i<\delta_i$, $\sigma(T)(i)=i+1$ and all $i's$ in $T$ are in row $i+1$.
\item if $\beta_i>s+\displaystyle\sum_{j=1}^{i-1}(\delta_j-\beta_j)$, then $\beta_i\geq\delta_i\geq\displaystyle
\sum_{j=1}^{i}\beta_j-\sum_{j=1}^{i-1}\delta_j-s$, $\sigma(T)(i)=\max\left\{k+1\mid k<i,\beta_i>s+\displaystyle\sum_{j=1}^{i-1}(\delta_j-\beta_j)\text{ or }k=0\right\}$ and all cells above row $i+1$ in $T$ are filled with $\{1,2,\dots,i\}$.
\item if $\beta_i=s+\displaystyle\sum_{j=1}^{i-1}(\delta_j-\beta_j)$, then
\begin{enumerate}
\item $\beta_i<\delta_i$, $\sigma(T)(i)=i+1$ and all $i's$ in $T$ are in row $i+1$. (same as case 1), or
\item $\delta_i=0$, $\sigma(T)(i)$ is the same as in case 2 and for all $i<j\leq n$, $\beta_j=\delta_j$, $\sigma(j)=j$ and all $j's$ are in row $j+1$ of $T$.
\end{enumerate}
\end{enumerate}
Let $Z_{\beta}^{\gamma}$ be the set of all integer vectors $\delta(T)=(s+1,\delta_1,\dots,\delta_n)$ that satisfy these three conditions and $\mathrm{comp}(\delta)=\gamma$, then we have $C_{\alpha,\beta}^{\gamma}=\displaystyle\sum_{\delta\in Z_{\beta}^{\gamma}}\mathrm{sgn}(\beta-\delta)$, where $\mathrm{sgn}(\beta-\delta)=(-1)^k$ and $k$ is the number of negative terms in $\beta-\delta$.
\end{Corollary}

\begin{proof}

(1) After filling $\{1,2,\dots,i-1\}$, if $\beta_i<s+\displaystyle\sum_{j=1}^{i-1}(\delta_j-\beta_j)$, that means even if all $i$'s appear strictly above row $r+1$, there are still some cells left unfilled. By a similar argument as in property $(1)$, if $k$ appears in the first $m$ rows of $T$, then $\sigma(k)\leq m$. An iterative use of property $(3)$ shows that either $\sigma(i)\leq i$, or $\sigma(T)=i+1$. Hence, if $\sigma(i)\leq i$, as we already have $\sigma(k)\leq i$ for all $k<i$, we cannot have any number larger than $i$ appear in the first $i$ rows of $T$. Therefore, all $i$'s must appear only in row $r+1$, $\sigma(i)=i+1$ and $\beta_i<\delta_i$.\\
\\
(2) After filling $\{1,2,\dots,i-1\}$, if $\beta_i>s+\displaystyle\sum_{j=1}^{i-1}(\delta_j-\beta_j)$. If $\sigma(i)=i+1$, by a similar argument to that in property $(2)$, we will get a most nefarious cell in row $i+1$ of $y(T)$, contrary to our choice of $T$. Hence, $\sigma(i)\leq i$. Since we already have $\sigma(k)\leq i$ for all $k<i$, $\sigma(i)$ is uniquely determined. To find $\sigma(i)$, we need to trace back and find the last time where $\sigma(k)\neq k+1$ among all $k<i$, or $1$ if $k$ does not exist. Therefore, $\sigma(T)(i)=\max\left\{k+1\mid k<i,\beta_i>s+\displaystyle\sum_{j=1}^{i-1}(\delta_j-\beta_j)\text{ or }k=0\right\}$. By a similar argument to that in property $(2)$, we must have $\beta_i\geq\delta_i$. And $\delta_i\geq\displaystyle
\sum_{j=1}^{i}\beta_j-\sum_{j=1}^{i-1}\delta_j-s$ because we need enough space to put $\{1,2,\dots,i\}$ into the first $i+1$ rows of $T$.\\
\\
(3) After filling $\{1,2,\dots,i-1\}$, if $\beta_i=s+\displaystyle\sum_{j=1}^{i-1}(\delta_j-\beta_j)$, then we have two cases. If $\sigma(i)=i+1$, then we are in the same situation as in (1). If $\sigma(i)\leq i$, then all $i$'s appear strictly above row $i+1$ and $\delta_i=0$. By a similar argument as in property (4), we have $\beta_j=\delta_j$ for all $i<j$.\\
\\
To sum up, each time we have a $\beta_i<\delta_i$, that corresponds to a $\sigma(i)=i+1$. Therefore, the sign of $\sigma(T)$ is sgn$(\beta-\delta)$.
\end{proof}

\noindent
Clearly, if $\ell(\gamma)<\ell(\beta)$ or $\ell(\gamma)>\ell(\beta)+1$, then $Z_{\beta}^{\gamma}=\emptyset$. If $\ell(\gamma)=\ell(\beta)+1$, then either $Z_{\beta}^{\gamma}=\emptyset$ or $Z_{\beta}^{\gamma}=\{\gamma\}$. If $\ell(\gamma)=\ell(\beta)$, it could happen that $\delta\neq\delta'$ but comp$(\delta)=$comp$(\delta')$.\\
\\
Suppose $\delta=(\delta_1,\dots,\delta_n)\in Z_{\beta}^{\gamma}$ and $\ell(\delta)=\ell(\beta)$. Let $1\leq k\leq n$ be the smallest integer such that $\beta_j=\delta_j$ for all $j>k$. Let $k\leq r\leq n$ be the largest integer such that $\beta_j<\beta_{j+1}$ for all $k\leq j<r$.\\
\\
Since $\delta\in Z_{\beta}^{\gamma}$, the composition $(\delta_1,\dots,\delta_k,0,\delta_{k+1},\dots,\delta_n)$ will satisfy the conditions in Corollary \ref{cor:5.9}. If $\beta_k<\beta_{k+1}=\delta_{k+1}$, by condition (3), we can always interchange $\sigma(k)$ to $k+1$ and obtain $(\delta_1,\dots,\delta_{k+1},0,\delta_{k+2},\dots,\delta_n)$ which also satisfies the conditions in Corollary \ref{cor:5.9}. However, for $j>r$, we cannot have the composition because condition (3.$a$) fails at row $r$. Since the compositions proceed in alternating signs, they cancel each other in pairs. Therefore, if $r-k$ is odd, then everything cancels and if $r-k$ is even, the first composition is left.\\
\\
Finally, we have the following criterion to determine the structure constants, and Lemma \ref{lem:5.8} and Corollary \ref{cor:5.9} give an algorithm to construct the corresponding tableau. At each step of filling numbers, if Corollary \ref{cor:5.9} fails, then Lemma \ref{lem:5.8} gives a corresponding $\Phi_r$ that cancels it, or an implicit involution using $\fS_0$. Therefore, combining with the argument above, we can assign each tableau in the above cases their corresponding involutions, and the remaining tableaux in $\mathfrak{T}_{\alpha}^{\beta}$ are left fixed. This indeed gives an involution because all the cases are disjoint and the involution sends the tableaux into tableaux in the same case, clear from the construction in Lemma \ref{lem:5.8}.\\
\\
For each outer shape $\gamma$, there can be at most one tableau fixed by this involution, which shows the left-Pieri rule is multiplicity free.

\begin{Theorem}
For $\alpha=(1)$ and a partition $\beta$ of length $n$, $\fS_{\alpha}\fS_{\beta}=\displaystyle\sum_{\gamma}C_{\alpha,\beta}^{\gamma}\fS_{\gamma}$ and
\begin{align*}
C_{\alpha,\beta}^{\gamma}=\left\{
\begin{array}{l l}
\mathrm{sgn}(\beta_1-\gamma_2,\dots,\beta_n-\gamma_{n+1})&\text{if }\ell(\gamma)=\ell(\beta)+1\text{ and }\gamma\in Z_{\beta}^{\gamma}\\
\\
\mathrm{sgn}(\beta_1-\gamma_2,\dots,\beta_{k-1}-\gamma_k)&\text{if }\ell(\gamma)=\ell(\beta),r-k\text{ is even and }\\&(\gamma_1,\dots,\gamma_k,0,\gamma_{k+1},\dots,\gamma_n)\in Z_{\beta}^{\gamma}\\
\\
0&otherwise
\end{array}\right.
\end{align*}
where $k$ and $r$ are as defined above.
\end{Theorem}

\begin{Example}
Let $\alpha=(1)$, $\beta=(3,1,4)$ and $\gamma=(2,3,2,2)$. As shown below, since $\gamma_1-1<\beta_1$ and $\gamma_2\leq\beta_1$, we have $\sigma(1)=1$ and all $1$'s appear in rows 1 and 2 of $T$. Then, since $\gamma_2-2=\beta_2$ and $\gamma_3>\beta_2$, we have $\sigma(2)=3$ and all $2$'s appear in row 3 of $T$. Finally, $\sigma(3)=2$. Therefore, $\sigma=(1,3,2)$ and $C_{\alpha,\beta}^{\gamma}=-1=\mathrm{sgn}(1,-1,2)$.

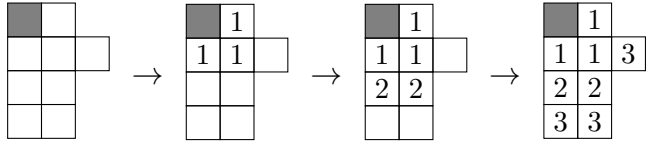
\[ \tikztableausmall{{\boldentry X, },{ , , },{ , },{ , }} \to \tikztableausmall{{\boldentry X, 1},{1,1, },{ , },{ , }} \to \tikztableausmall{{\boldentry X, 1},{1,1, },{2,2},{ , }} \to \tikztableausmall{{\boldentry X, 1},{1,1,3},{2,2},{3,3}} \]
\end{Example}

\noindent
This result is equivalent to the one in \cite{BSZ}, but here we give an explicit combinatorial interpretation and an algorithm for constructing the tableaux corresponding to the structure constants.

\end{document}